\documentclass[11pt]{amsart}
\usepackage{geometry}                
\usepackage{graphicx}
\usepackage{amssymb}
\usepackage{epstopdf}
\usepackage{amscd,amsfonts,amsmath,amssymb,amsthm,bbm,bbold,enumerate,epsf,fancyhdr,float,graphicx,latexsym,mathrsfs,multirow,smfthm,verbatim,wasysym,xypic} \usepackage[all]{xy}\usepackage[OT2,T1]{fontenc}
\usepackage[modulo,mathlines,displaymath]{lineno}
\DeclareSymbolFont{cyrletters}{OT2}{wncyr}{m}{n}\DeclareMathSymbol{\Sha}{\mathalpha}{cyrletters}{"58}
\newcommand{\Q}{\mathbb{Q}}\newcommand{\C}{\mathbb{C}}\newcommand{\Z}{\mathbb{Z}}\newcommand{\R}{\mathbb{R}}
\newcommand{\links}{\left(\begin{array}{cc}}\newcommand{\rechts}{\end{array}\right)}\newcommand{\bai}{\left[\begin{array}{cc}}\newcommand{\dai}{\end{array}\right]}\newcommand{\hidari}{\left(\begin{array}{c}}\newcommand{\migi}{\end{array}\right)}\newcommand{\Reg}{{\mathrm{Reg}}}\newcommand{\CCC}{{\mathcal C}}\newcommand{\rank}{{\mathrm{rank}}}\newcommand{\ord}{{\mathrm{ord}}}\newcommand{\Sel}{{\mathrm{Sel}}}\renewcommand{\O}{{\mathcal{O}}}\newcommand{\Fp}{\mathbf{F}_p}\renewcommand{\geq}{\geqslant}\renewcommand{\leq}{\leqslant} \newcommand{\smat}[1]{\left( \begin{smallmatrix} #1 \end{smallmatrix} \right)}\renewcommand{\qedsymbol}{{$\mathit{QED}$}}\renewcommand{\phi}{{\varphi}}\newcommand{\Log}{\mathcal Log}\newcommand{\nuek}{{\nu_{A}}}\newcommand{\nudo}{{\nu_{B}}}
\newtheorem{theorem}{Theorem}[section]\newtheorem{conjecture}[theorem]{Conjecture}\newtheorem{corollary}[theorem]{Corollary}\newtheorem{lemma}[theorem]{Lemma}\newtheorem{proposition}[theorem]{Proposition}\newtheorem{open problem}[theorem]{Open Problem}\theoremstyle{definition}\newtheorem{definition}[theorem]{Definition}\newtheorem{convention}[theorem]{Convention}\newtheorem{notation}[theorem]{Notation}\newtheorem{remark}[theorem]{Remark}

\renewcommand{\qedsymbol}{Q.E.D.}
\newtheorem{question}{Question}

\title[Supersingular $p$-adic BSD]{A formulation for $p$-adic versions of the Birch and Swinnerton-Dyer Conjectures in the supersingular case}
\author{Florian E. I. Sprung}

\begin{document}
\maketitle
\section{Introduction}
Let $E$ be an elliptic curve over the rational numbers $\Q$. The $\Q$-rational points form a finitely generated abelian group $E(\Q)=\Z^r\oplus E(\Q)_{tors}$, and the classical Birch and Swinnerton-Dyer conjectures predict that the order of vanishing $r^{an}_\C$ of the Hasse-Weil $L$-function $L(E,s)$ at $s=1$, an analytic quantity, should be equal to $r$, an algebraic quantity. The second part of their conjecture says that the leading Taylor coefficient of $L(E,s)$ should encode $E(\Q)_{tors}$ and the size of the Tate-Shafarevich group $\Sha(E/\Q)$, among other algebraic quantities. More precisely, this conjecture says the following.

Denote by $\omega=\omega_E$ the N\'{e}ron differential and by $\Omega_E=\int_{E(\R)}\omega\in\R^{>0}$ the N\'{e}ron period of $E$. 
We denote by $L^*(E)$ the leading coefficient of the Taylor expansion at $s=1$.

\begin{conjecture}[BSD]\quad

\begin{enumerate}

\item We have $r^{an}_\C=r$.
\item $\frac{L^*(E)}{\Omega_E}=\frac{\prod_v c_v \cdot \#\Sha(E/\Q)\cdot \Reg_\C(E/\Q)}{(\#E(\Q)_{tors})^2}$.
\end{enumerate}
Here,  $c_v$ denotes the Tamagawa number for a place $v$, and the regulator $\Reg_\C(E/\Q)$ is the discriminant of the N\'{e}ron-Tate canonical height pairing on $E(\Q)$.
\end{conjecture}

To construct a $p$-adic  analogue, we identify algebraic numbers with $p$-adic numbers by  fixing an embedding $\overline{\Q}\hookrightarrow \C_p$. A $p$-adic analogue of this conjecture should look as follows. There should be a $p$-adic $L$-function $L_p(E,T)$, whose order of vanishing $r^{an}$ at $T=0$ should equal $r$, and whose leading Taylor coefficient should again encode algebraic quantities of $E$ including $E(\Q)_{tors}$ and $\#\Sha(E/\Q)$ as in the above formula, with the regulator $\Reg_\C$ replaced by a $p$-adic avatar. There are two types of such $p$-adic analogues.

The first type concerns the case when $p$ is a prime of ordinary reduction (meaning that $p$ is coprime to $a_p:=p+1-\#E(\Fp)$). Here, Mazur, Tate, and Teitelbaum formulated a $p$-adic version of these conjectures. The $p$-adic $L$-function $L_p(E,T)$ they employed gave rise to the expected version of $p$-adic BSD when $p$ was of good ordinary reduction, but in the split multiplicative case $r^{an}$ corresponded to $r+1$ in view of an extra zero. 

The second type is the supersingular case (the case when $p|a_p$), which is more complex. There are \textit{two} classical $p$-adic $L$-functions, denoted $L_\alpha(E,T)$ and $L_\beta(E,T)$, constructed independently by Amice and V\'{e}lu, and  Vishik. The subscripts $\alpha$ and $\beta$ denote the roots of the Hecke polynomial $Y^2-a_pY+p$. The formulation of a $p$-adic analogue of the Birch and Swinnerton-Dyer conjectures in terms of these $L_\alpha(E,T)$ and $L_\beta(E,T)$ when $p$ is of good reduction is due to Bernardi and Perrin-Riou. For questions of formulating Birch and Swinnerton-Dyer conjectures, this seems to suggest that the picture is complete except for some primes of bad reduction. The goal of this paper is to indicate that this is not the case.

We do this by formulating $p$-adic versions of BSD for supersingular primes $p$ using a more natural pair of $p$-adic $L$-functions $L_\sharp(E,T)$ and $L_\flat(E,T)$. This hints at a formulation of $p$-adic BSD in the ordinary case in terms of such a pair as well. These $p$-adic $L$-functions had been constructed by Pollack, Kobayashi, and the author in the supersingular case and are functions living in the power series ring $\Z_p[[T]]$, unlike $L_\alpha(E,T)$ and $L_\beta(E,T)$ which have more complicated growth properties. Apart from being an ingredient for a natural formulation of the Iwasawa Main Conjecture, the appearance of the Iwasawa invariants of this pair of $p$-adic $L$-functions indicates that this pair is the natural choice: The Iwasawa invariants appear in analytic estimates for the sizes of the $p$-primary parts of the Tate-Shafarevich group along the cyclotomic $\Z_p$-extensions. The term ``analytic estimates" refers to the corresponding special values of the Hasse-Weil $L$-functions twisted by characters of $p$-power conductor, which should encode these sizes. These \textit{analytic} estimates can be found in \cite{pollack} and \cite
{surprisesha}. There are algebraic counterparts of $L_\sharp(E,T)$ and $L_\flat(E,T)$, which are modified Selmer groups $\Sel^\sharp(E)$ and $\Sel^\flat(E)$. Using their Iwasawa invariants one can estimate these sizes directly, see e.g. \cite{kobayashi} and \cite{sha3}. In addition to these estimates, the Iwasawa invariants appear in upper bounds for the rank of the elliptic curve in the cyclotomic $\Z_p$-extension of $\Q$. For analytic estimates, see \cite{pollack} and \cite{surprisesha}, and for their algebraic counterparts, see \cite{kobayashi} and \cite{rank}.

Concretely, our conjecture says the following.

\begin{conjecture}[Tandem $p$-adic BSD]Let $E$ be an elliptic curve and $p$ a prime of good supersingular reduction.
Denote by $\overrightarrow{L}_p^*$ the first non-zero Taylor coefficient around $T=0$ of the vector of $p$-adic $L$-functions $(L_\sharp(E,T),L_\flat(E,T))$.
\begin{enumerate}
\item The minimum of the orders of vanishing of $L_\sharp(E,T)$ and $L_\flat(E,T)$ at $T=0$ is equal to  $r$.
\item $\overrightarrow{L}_p^*=\frac{\prod_v c_v \cdot \#\Sha(E/\Q)}{(\#E(\Q)_{tors})^2} \Reg_p^\natural(E/\Q).$
\end{enumerate}\end{conjecture}

Our regulator $\Reg_p^\natural(E/\Q)$ is constructed explicitly from height functions that mirror the construction of $L_\sharp(E,T)$ and $L_\flat(E,T)$. 

\begin{theorem}This conjecture is equivalent to the conjectures of Bernardi and Perrin-Riou. \end{theorem}

Since we have two functions at hand, we may consider their quotient, and give the following criterion for detecting non-zero rank:

\begin{theorem} Assume property $(*)$ below holds and that $\Sha(E/\Q)[p^\infty]<\infty$. Then
$$ \rank{E(\Q)}>0 \iff \left. \frac{L_\sharp(E,T)}{L_\flat(E,T)} \right|_{T=0}\neq \frac{-a_p^2+2a_p+p-1}{2-a_p} \text{ for odd $p$, and } $$
$$ \rank{E(\Q)}>0 \iff \left. \frac{L_\sharp(E,T)}{L_\flat(E,T)} \right|_{T=0}\neq \frac{-a_2^3+2a_2^2+3a_2-4}{-a_2^2+2a_2+1} \text{ for $p=2$}. $$\end{theorem}

This theorem is a generalization of \cite[Corollary 0.5]{kuriharapollack} who assumed $p$ to be odd and $a_p=0$, which is automatically satisfied when $p\geq5$. We remark that the proof of Kurihara and Pollack almost immediately generalizes once the $L_\sharp(E,T)$ and $L_\flat(E,T)$ have been defined in complete generality, as done in \cite{shuron}. There is only one proposition in their tools that assumes $a_p=0$, which is fixed in this paper by adhering to a generalization found in \cite{sha3}.

As a corollary to this, we get:
\begin{corollary}Let $\rank E(\Q)=0$. Then both $L_\sharp(E,T)$ and $L_\flat(E,T)$ are non-zero functions, confirming \cite[Conjecture 6.15]{shuron} in this case.
\end{corollary}

Another object to consider given two quantities is their greatest common divisor. Denote by $d_n$ the normalized jump in ranks $\frac{1}{p^n-p^{n-1}}\left(\rank E(\Q_n)- \rank E(\Q_{n-1})\right)$, where $\Q_n$ is the $n$th layer in the cyclotomic $\Z_p$-extension numbered so that $\Q_0=\Q$. We also let $\Phi_{p^n}$ be the $p^n$th cyclotomic polynomial. Kurihara and Pollack have formulated the following conjecture in the case $a_p=0$ and $p$ odd:

\begin{conjecture}Let $E/\Q$ be an elliptic curve, and $p$ a good supersingular prime. Then
$$\gcd(L_\sharp(E,T),L_\flat(E,T))=\left(T^{r}\prod_{ d_n \geq1 \text{ and } n\geq 1}\Phi_{p^n}^{d_n-1}(1+T)\right).$$
\end{conjecture}

Note that we excluded their assumptions. This is because we give some evidence towards their conjecture by proving the following proposition which works for general supersingular $p$. Denote by $r^{an}$ the order of vanishing of $L(\alpha,T)$ (or $L(\beta,T)$) at $T=0$:

\begin{proposition}\label{kuriharapollack}Let $E/\Q$ be an elliptic curve and $p$ a prime of good reduction. For some polynomial $P_\Sha(E,T)$ with $P_\Sha(E,\zeta_{p^n}-1)\neq0$ for $n\geq0,$
$$\gcd\left(L_\sharp(E,T),L_\flat(E,T)\right)=\left(P_\Sha(E,T)\cdot T^{r^{an}}\prod_{ d_n \geq1 \text{ and } n\geq 1}\Phi_{p^n}^{\epsilon_n^{an}-1}(1+T)\right),$$
where $\epsilon_n^{an}=d_n^{an}$ or $\epsilon_n^{an}=d_n^{an}+1.$
\end{proposition}

Finally, we ask a question on a possible generalization of the situation to the ordinary case. When $p$ is ordinary, the functions $L_\sharp(E,T)$ and $L_\flat(E,T)$ are not uniquely defined, but their values at $T=0$ are. We find that when $p$ is odd and $a_p=2$, $L_\flat(E,0)=0$.

\begin{question} Where does this `extra zero phenomenon' in the ordinary case come from?
\end{question}

In \cite{loefflerzerbes},  Loeffler and Zerbes find a pair of Iwasawa functions in the ordinary case using the theory of Wach modules. The above question suggests that an extra zero phenomenon should occur in terms of their pair as well.

\section{Review of $p$-adic $L$-functions for elliptic curves}

Let $p$ be a prime of good reduction for our elliptic curve $E$. The work of Mazur, Swinnerton-Dyer, Amice and V\'{e}lu, and Vi\v{s}ik gives constructions of $p$-adic $L$-functions that should encode the behavior of the $\Q_n$-rational points $E(\Q_n)$ of the elliptic curve $E$. Concretely, denote by $\alpha$ and $\beta$ the roots of the Hecke polynomial $Y^2-a_pY+p$ ordered so that $\ord_p(\alpha)\leq\ord_p(\beta)$ for any $p$-adic valuation $\ord_p$. We say that $\alpha$ resp. $\beta$ is an allowable root if $\ord_p(\alpha)<\ord_p(p)$ resp. $\ord_p(\beta)<\ord_p(p)$. Notice that $\alpha$ is always allowable by convention, while $\beta$ is only when $p$ is supersingular.

\notation We denote by $\zeta_{p^n}$ a primitive $p^n$th root of unity, and we let $N=n+1$ when $p$ is odd and $N=n+2$ when $p=2$. We also let $\chi_u$ be a group morphism from $1+2p\Z_p$ into $\C_p^\times$ sending a topological generator $1+2p$ to some $u\in\C_p$ so that $|u-1|_p<1$, where $|\quad |_p$ is the normalized $p$-adic absolute value (i.e. $|\frac{1}{p}|_p=1$). 

\begin{theorem}[Mazur and Swinnerton-Dyer, Amice and V\'{e}lu, Vi\v{s}ik]\cite[Proposition in Section 14]{mtt} Let $E$ be an elliptic curve over $\Q$ and $p$ a prime of good reduction. Regard $\chi_{\zeta_{p^n}}$ and $\chi_{\zeta_{p^n}^{-1}}$ as characters of $\Z_p/p^n\Z_p.$ There is a $p$-adic analytic function $L_\alpha(E,T)$ converging on the open unit disk with the following interpolation properties:
$$L_\alpha(E,\zeta_{p^n}-1)=\frac{p^N}{\alpha^N\tau(\chi_{\zeta_{p^n}^{-1}})}\frac{L(E,\chi_{\zeta_{p^n}^{-1}},1)}{\Omega_E} , and$$
$$ L_\alpha(E,0)=\left(1-\frac{1}{\alpha}\right)^2\frac{L(E,1)}{\Omega_E}.$$
When $\beta$ is allowable (i.e. $p$ is supersingular), there is a companion $p$-adic analytic function $L_\beta(E,T)$ with the same interpolation properties with $\alpha$ replaced by $\beta$.
\end{theorem}

While these classical $L_\alpha(E,T)$ have bounded coefficients when $p$ is ordinary, they are not elements of $\Z_p[[T]]$ (or even of $\Z_p[[T]]\otimes \overline{\Q}_p$) when $p$ is supersingular. The following theorem fixes this situation:

 \begin{theorem}\label{maintheorem}\cite[Theorem 5.6 for $a_p=0$]{pollack}\cite[Theorem 6.12 for general $p|a_p$]{shuron} Let $p$ be a prime of good supersingular reduction. Then there are two $p$-adic $L$-functions $L_\sharp(E,T)$ and $L_\flat(E,T)$ which are elements of $\Z_p[[T]]$ so that we may write
$$(L_\alpha(E,T),L_\beta(E,T))=(L_\sharp(E,T), L_\flat(E,T))\Log_{\alpha,\beta}(1+T),$$
where $$\Log_{\alpha,\beta}(1+T):=\lim_{n \rightarrow \infty}\smat {a_p & 1 \\ \Phi_{p^1}(1+T) & 0 }\smat {a_p & 1 \\ \Phi_{p^2}(1+T) & 0 }\cdots \smat {a_p & 1 \\ \Phi_{p^n}(1+T) & 0 }\smat{ a_p & 1 \\ p & 0}^{-(N+1)}\smat{-1 & -1\\  \beta & \alpha}$$
is a matrix of $p$-adic analytic functions converging on the open $p$-adic unit disk.
\end{theorem}

\begin{theorem}\cite[Theorem 2.14]{surprisesha} Let $p$ be an ordinary good prime. Then the first column $\log_\alpha$ of $\Log_{\alpha,\beta}(1+T)$ converges and we have $$L_\alpha(E,T)=(L_\sharp(E,T), L_\flat(E,T))\log_\alpha, $$ for two $p$-adic analytic functions converging on the closed $p$-adic unit disk $L_\sharp(E,T)$ and  $L_\flat(E,T)$.
\end{theorem}

We remark that these $L_\sharp(E,T)$ and  $L_\flat(E,T)$ are not uniquely defined in the ordinary case, but their values at $T=0$ are unique:

\begin{lemma}The value of the vector $\left(L_\sharp(E,0),L_\flat(E,0)\right)$ equals\begin{linenomath}$$\begin{cases}(-a_p^2+2a_p+p-1,-a_p+2)\cdot\frac{L(E,1)}{\Omega_E}& \text{ when $p$ is odd,}\\(-a_p^3+2a_p^2+2pa_p-a_p-2p,-a_p^2+2a_p+p-1)\cdot\frac{L(E,1)}{\Omega_E}& \text{ when $p$ is even.}
\end{cases}$$\end{linenomath}
\end{lemma}
\begin{proof}This follows from the table before Proposition 6.14. in \cite{shuron} in the supersingular case and the one before Conjecture 5.18 in \cite{surprisesha} for the general case.
\end{proof}

\section{The behavior at $T=0$: A criterion for non-zero rank}
The goal of this short section is to prove:

\begin{theorem} Assume property $(*)$ below holds and that $\Sha(E/\Q)[p^\infty]<\infty$. Then
$$ \rank{E(\Q)}>0 \iff \left. \frac{L_\sharp(E,T)}{L_\flat(E,T)} \right|_{T=0}\neq \frac{-a_p^2+2a_p+p-1}{2-a_p} \text{ for odd $p$, and } $$
$$ \rank{E(\Q)}>0 \iff \left. \frac{L_\sharp(E,T)}{L_\flat(E,T)} \right|_{T=0}\neq \frac{-a_2^3+2a_2^2+3a_2-4}{-a_2^2+2a_2+1} \text{ for $p=2$}. $$
\end{theorem}

Let $T_p(E)$ be the Tate module for $E$. 

Here is property $(*)$: The composite of natural maps
$$ \mathbf{H}^1_{\text{glob}}=\varprojlim H^1_{\text{\'{e}t}}(\O_{\Q_n}[1/S],T_p(E))\rightarrow H^1(\Q,T_p(E))\rightarrow H^1(\Q_p,T_p(E))$$

is not zero. Here, the limit is taken with respect to corestriction, and $S$ is the product of the bad reduction primes and $p$. 

We remark that property $(*)$ should always be true.

\begin{proof}[Proof of Theorem]The arguments of Kurihara and Pollack of \cite[Section 1.4]{kuriharapollack} almost work with the appropriate modifications. In terms of notation, we write $Col(z)=(Col^\sharp(z),Col^\flat(z))$ with the Coleman maps from \cite{shuron} instead of the functions $h_z(T)$ and $k_z(T)$. We note that their arguments work in spite of having to take into account the possibility of one of $L_\sharp(E,T)$ or $L_\flat(E,T)$ to be zero, cf. \cite[6.15]{shuron} and the surrounding discussions! Finally, \cite[Proposition 1.2]{kuriharapollack} is only proved in the case $a_p=0$. To remedy this, we refer to the isotypical component of the trivial tame character of the inverse limit of \cite[Proposition 4.7]{sha3} and remark that the arguments found therein all apply when $p=2$ as well.\end{proof}

\begin{corollary} \cite[Conjecture 6.15]{shuron} said that both $L_\sharp(E,T)$ or $L_\flat(E,T)$ are non-zero. This conjecture holds in the case where $\rank E(\Q)=0$.
\end{corollary}
\begin{proof}We know that at least one of $L_\sharp(E,T)$ or $L_\flat(E,T)$ is non-zero by \cite[Proposition 6.14]{shuron}, so the theorem tells us they both have to be.
\end{proof}

\section{The behavior at $T=0$: $p$-adic versions of the BSD conjectures}
The goal of this section is to formulate $p$-adic versions of BSD in terms of the vector $(L_\sharp(E,T),L_\flat(E,T))$ in the supersingular case. We thus assume that $p$ is supersingular for this section. 

\subsection{Dieudonn\'{e} modules and $p$-adic heights}
 The Dieudonn\'{e} module is the following two-dimensional $\Q_p$-vector space:
\begin{linenomath}$$D_p(E):=\Q_p\otimes H^1_{dR}(E/\Q)$$\end{linenomath}
There is a Frobenius endomorphism $\phi$ which acts on $D_p(E)$ linearly. We refer the reader to \cite[Paragraph 2]{bpr} for a concrete definition, but let us record that its characteristic polynomial is $Y^2-\frac{a_p}{p}Y+\frac{1}{p}$, as opposed to the definition of \cite{mazursteintate} or \cite{kedlaya} (where it is $Y^2-a_pY+p$). This vector space admits a basis $\omega$ and $\phi(\omega)$, where $\omega$ is the invariant/N\'{e}ron differential of $E$.
We want to define eigenvectors $\nuek:=\nu_{\frac{1}{\alpha}}$ and $\nudo:=\nu_{\frac{1}{\beta}}$ of $\phi$ with eigenvalues $\frac{1}{\alpha}$  and $\frac{1}{\beta}$ which live in the $\Q_p(\alpha)$-vector space
\begin{linenomath}$$D_p(E)(\alpha):=\Q_p(\alpha)\otimes H^1_{dR}(E/\Q).$$\end{linenomath}
\begin{definition}We define (scale) both eigenvectors as follows:
\begin{linenomath}$$\hidari \nuek \\\nudo \migi :=\links -\alpha & p \\ \beta & -p \rechts \frac{1}{\beta-\alpha}\hidari \omega \\ \phi(\omega) \migi$$\end{linenomath}
\end{definition}
\begin{definition}Perrin-Riou's $p$-adic $L$-function can be defined via the classical $p$-adic $L$-functions $L_\alpha:=L_\alpha(E,\alpha,T)$ and $L_\beta:=L_p(E,\beta,T)$:
\begin{linenomath}$$L_p^{PR}(E,T):=(L_\alpha,L_{\beta})\hidari \nuek \\ \nudo \migi$$\end{linenomath}
\end{definition}
This is equivalent via the arguments in \cite[Section 3.5]{steinwuthrich} to Perrin-Riou's construction in \cite[Section 2.2]{perrinriou}.
\begin{lemma}[$D_p(E)$-rationality of coefficients] We have $L_p^{PR}(E,T) \in D_p(E)[[T]]$.
\end{lemma}
\begin{proof}By Theorem \ref{maintheorem}, we can write
\begin{linenomath}$$L_p^{PR}(E,T)=(L_\sharp, L_\flat)\Log_{\alpha,\beta}\smat{1&\alpha\\1&\beta}^{-1}\smat{1 & 0\\a_p & -p}\smat{\omega \\ \phi(\omega)}.$$\end{linenomath}
From the definition of $\Log_{\alpha,\beta}$, we then see that $L_p^{PR}(E,T)\in D_p(E)[[T]]$, as desired.
\end{proof}
Given a globally minimal Weierstrass equation over $\Z$
\begin{linenomath}$$y^2+a_1xy+a_3y=x^3+a_2x^2+a_4x+a_6$$\end{linenomath} for $E$, recall that the associated N\'{e}ron/invariant differential is $\omega=\frac{dx}{2y+a_1x+a_3}$ (see e.g. \cite[Chapter III.1]{silverman}). The $\Q$-vector space $H^1_{dR}(E/\Q)$ admits a basis $\{\omega,x\omega\}$ and is equipped with a canonical alternating bilinear form $[\cdot,\cdot]$ so that ${[\omega,x\omega]=1}$. We extend it linearly to the Dieudonn\'{e} modules above and denote these extensions by $[\cdot,\cdot]$ as well.

Fix $\omega$. Then for each $\nu\in D_p(E)$ (resp. $\nu\in D_p(E)(\alpha)$), one can associate a quadratic form $h_\nu$ mapping $E(\Q)$ to $\Q_p$ (resp. to $\Q_p(\alpha))$. One can do this (see \cite{bpr}, or \cite{steinwuthrich}) by defining preliminary height functions $h'_{\omega}$ and $h'_{x\omega}$, and then extending linearly, i.e. given $\nu=a\omega+bx\omega$, let $h'_\nu=ah'_\omega+bh'_{x\omega}$. Explicitly, we have $h'_\omega(P)=-\log_{\omega}(P)^2$, where $\log_{\omega}$ is the logarithm associated to $\omega$. The definition of $h'_{x\omega}$ involves the $\sigma$-functions of either Mazur and Tate or of Bernardi. We refer to \cite[Section 4]{steinwuthrich} for an explicit definition, since it won't be needed in this paper. We then normalize and put $h_\nu:=\frac{h'_\nu}{\log_p(\gamma)}$.
\remark The reason for this normalization is that $p$-adic $L$-functions $L_p(T)$ are  classically thought of as functions of a variable $s$ via the substitution $T=\gamma^{s-1}-1$, cf. \cite[\S II.13]{mtt}. The original formulation of $p$-adic BSD then investigated the behavior of a particular $L_p(\gamma^{s-1}-1)$ at $s=1$. Note that \begin{linenomath}$$\frac{d^r}{ds^r}L_p(\gamma^{s-1}-1)|_{s=1}=\left.\frac{d^r}{dT^r}L_p(T)\right|_{T=0}\cdot\log_p(\gamma)^r.$$\end{linenomath}

 The bilinear form associated to this height function has values in $\Q_p$ (resp.$\Q_p(\alpha)$):
 \begin{linenomath}$$\langle P,Q\rangle_\nu=\frac{1}{2}\left(h_\nu(P+Q)-h_\nu(P)-h_\nu(Q)\right)$$\end{linenomath}
\definition Let $\Reg_\nu$ be the discriminant of this height pairing on $E(\Q)/E(\Q)_{tors}$.

\begin{definition}Let $\nuek$ and $\nudo$ be as above. We define normalized height functions
\begin{linenomath}$$\hat{h}_\nuek:=\frac{h_\nuek}{\left[\nudo,\nuek\right]}=\frac{h_\nuek}{\left[\omega,\nuek\right]} \text{ and } \hat{h}_\nudo:=\frac{h_\nudo}{\left[\nuek,\nudo\right]}=\frac{h_\nudo}{\left[\omega,\nudo\right]},$$\end{linenomath}
and denote their corresponding regulators by $\Reg_{\frac{1}{\alpha}}$ and $\Reg_{\frac{1}{\beta}}$. Note that the
height functions (and thus the regulators) are independent of the choice of our Weierstra\ss\text{ }equation.
\end{definition}

Denote by $r(E)$ the rank of $E(\Q)$. In the supersingular case, Perrin-Riou defines the regulator $\Reg_p^{BPR}(E/\Q)$ as the unique element in $D_p(E)$ so that for any $\nu\in D_p(E)$ with $\nu\not\in\Q_p\omega,$ we have \footnote{This characterization is that of \cite[Lemma 4.2]{steinwuthrich}, which is a corrected version of Perrin-Riou's original lemma, \cite[Lemme 2.6]{perrinriou}.}
\begin{equation}\label{universal}
\left[\Reg_p^{BPR}(E/\Q),\nu\right]=\frac{\Reg_\nu}{[\omega,\nu]^{r-1}}\text{ where $r=r(E)>0$}.
\end{equation} For $r(E)=0$, she puts $\Reg_p^{BPR}(E/\Q)=\omega.$

\begin{definition}As an element of $D_p(E)(\alpha)$, define \begin{linenomath}$$\Reg_p(E/\Q):=( \Reg_{\frac{1}{\beta}},\Reg_{\frac{1}{\alpha}})\hidari \nuek \\ \nudo \migi.$$\end{linenomath}
\end{definition}

\begin{proposition}\label{computationofregulator}We have $\Reg_p^{BPR}(E/\Q)=\Reg_p(E/\Q)\in D_p(E)$. 
\end{proposition}
\begin{proof}

When $r(E)=0$, this follows from the fact that $\Reg_{\frac{1}{\alpha}}=1$ and $\Reg_{\frac{1}{\beta}}=1$.

When $r(E)>0$, note that
\begin{linenomath}$$\left[\Reg_p(E/\Q),\nuek\right]=\Reg_{\frac{1}{\alpha}}\cdot \left[\nudo,\nuek\right]=\frac{\Reg_\nuek}{\left[\nudo,\nuek\right]^{r-1}}=\frac{\Reg_\nuek}{\left[\omega,\nuek\right]^{r-1}},$$\end{linenomath}
and similarly, $\left[\Reg_p(E/\Q),\nudo\right]=\frac{\Reg_\nudo}{\left[\omega,\nudo\right]^{r-1}}$. Since $\nuek$ and $\nudo$ form a basis for $D_p(E)(\alpha)$, linearity tells us that the property described in equation (\ref{universal}) holds for any $\nu\in D_p(E)(\alpha)$.

Now suppose that $\Delta:=\Reg_p(E/\Q)- \Reg_p^{BPR}(E/\Q)\neq 0$. Then
\begin{linenomath}$$\left[\Delta,\nu\right]=0 \text{ for any }\nu\in D_p(E),$$\end{linenomath}
so this would in particular hold for $\nu=\omega$ or $\nu=x\omega$, from which we conclude by linearity that
\begin{linenomath}$$\left[\Delta,\nu\right]=0 \text{ for any  }\nu\in D_p(E)(\alpha).$$\end{linenomath}
But this would imply $\Delta=0$. \renewcommand{\qedsymbol}{Q.E.A.}
\end{proof}
 \renewcommand{\qedsymbol}{Q.E.D.}

\subsection{Statement of the conjectures}


The following is a $p$-adic analogue of the Birch and Swinnerton-Dyer conjectures when $p$ is supersingular (cf. \cite[Conjecture on page 229]{bpr} and \cite[Conjecture 2.5]{perrinriou}):
\begin{conjecture}[Bernardi and Perrin-Riou]\label{bernardiandperrinriou} Let $p$ be a good supersingular prime, and denote by $r^{an}$ the order of vanishing of $L_p^{PR}(E,T)$ at $0$, and by $L_p^{PR*}(E)$ its leading coefficient (with value in the Dieudonn\'{e} module) of its Taylor expansion around $0$.
\begin{enumerate}
\item We have $r^{an}=r(E)$.
\item $L_p^{PR*}(E)=\left(1-\phi\right)^2\frac{\prod_v c_v \cdot \#\Sha(E/\Q)}{(\#E(\Q)_{tors})^2}\Reg_p(E/\Q)$.
\end{enumerate}
\end{conjecture}

\remark Note that while the objects in the second part of this conjecture depend on the choice of Weierstra\ss{ }equation, their coordinates with respect to the basis $\{\nu_\alpha,\nu_\beta\}$ don't. In fact, one can formulate Bernardi's and Perrin-Riou's conjecture in a form that resembles more closely that of the one given by Mazur, Tate, and Teitelbaum:

\begin{conjecture}[Equivalent formulation of above]\label{equivalentformulationofabove} Let $r^{an}$ be as in Lemma \ref{ordersofvanishingagree} below, and $L_\alpha^*$ and $L_\beta^*$ be the leading coefficients in the Taylor expansion. Then
\begin{enumerate}
\item $r^{an}=r(E)$.
\item $L_\alpha^*=(1-\frac{1}{\alpha})^2\frac{\prod_v c_v \cdot \#\Sha(E/\Q)\cdot \Reg_{\frac{1}{\beta}}}{(\#E(\Q)_{tors})^2}$, and $L_\beta^*=(1-\frac{1}{\beta})^2\frac{\prod_v c_v \cdot \#\Sha(E/\Q)\cdot \Reg_{\frac{1}{\alpha}}}{(\#E(\Q)_{tors})^2}$.
\end{enumerate}
\end{conjecture}
This version can be found in \cite[Conjecture 0.12]{colmez}, where it is attributed to Mazur, Tate and Teitelbaum.

\begin{lemma}\label{ordersofvanishingagree} Since $p$ is supersingular, $r^{an}:=\ord_{T=0} L_p(E,\alpha,T)=\ord_{T=0} L_p(E,\beta,T)$.
\end{lemma}
\proof The same proof as in \cite[Lemma 6.6]{pollack} works.

\subsection{A version of the conjectures via $L_\sharp$ and $L_\flat$ in the supersingular case}
We now reformulate Conjecture \ref{bernardiandperrinriou} using $L_\sharp$ and $L_\flat$. \begin{definition}\label{lvector} We call $\overrightarrow{L}_p(E,T):=\overrightarrow{L}_p:={(L_\sharp,L_\flat)}$ the \textit{$p$-adic $L$-vector of $E$}, and denote by $r_p^\natural$ the minimum of the orders of vanishing of $L_\sharp$ and $L_\flat$. \end{definition}

We would now like to find a pair of elements $\nu_\sharp$ and $\nu_\flat$ in $D_p(E)$ that give rise to regulators corresponding to our $p$-adic $L$-functions. Recall that $L_p^{PR}(E,T) = (L_\sharp, L_\flat) \Log_{\alpha,\beta}\hidari\nuek \\ \nudo \migi $.

\begin{definition} Let $Z:=\left.\Log_{\alpha,\beta}(1+T)\right|_{T=0}=\Log_{\alpha,\beta}(1)$. 
We define
\begin{linenomath}$$\hidari \nu_\sharp\\ \nu_\flat \migi :=Z\hidari\nuek \\ \nudo \migi,$$\end{linenomath}
\begin{linenomath}$$(N_\sharp, N_\flat) :=(\nu_B,-\nu_A)\smat{(1-\frac{1}{\alpha})^2 & 0 \\ 0 & (1-\frac{1}{\beta})^2}Z^{-1}\times \det Z.$$\end{linenomath}
\end{definition}

\begin{lemma}\label{nonlinearity}$\nu_\sharp, \nu_\flat, N_\sharp, N_\flat$ are in $D_p(E)$ and are not $\Q_p$-multiples of $\omega$.
\end{lemma}
\proof Calculation.

\begin{definition} We let $\Reg_\sharp:=\Reg_{\frac{N_\sharp}{[\omega,N_\sharp]}}$ and  $\Reg_\flat:=\Reg_{\frac{N_\flat}{[\omega,N_\flat]}}$ be the regulators for the normalized heights associated to $N_\sharp$ and $N_\flat$. Also, we let
\begin{linenomath}$$\Reg_p^\natural:=
\begin{cases}
\left(
\begin{array}{cc} (-a_p^2+2a_p+p-1)\Reg_\sharp,& (-a_p+2)\Reg_\flat \end{array}\right)
 & \text{ for odd $p$,}\\
  \left(\begin{array}{cc}(-a_p^3+2a_p^2+2pa_p-a_p-2p)\Reg_\sharp,&  (-a_p^2+2a_p+p-1)\Reg_\flat\end{array}\right)
 & \text{ for even $p$.}
\end{cases}
$$\end{linenomath}

\end{definition}

We are now ready to give our $p$-adic version of BSD:

\begin{conjecture}[Tandem $p$-adic BSD]\label{tandempadicbsd}Let $E$ be an elliptic curve and $p$ a prime of good supersingular reduction.
Denote by $\overrightarrow{L}_p^*$ the first non-zero leading Taylor coefficient around $T=0$ of $\overrightarrow{L}_p=\overrightarrow{L}_p(E,T)$.
\begin{enumerate}
\item We have $r_p^\natural=r(E)$.
\item $\overrightarrow{L}_p^*=\frac{\prod_v c_v \cdot \#\Sha(E/\Q)}{(\#E(\Q)_{tors})^2} \Reg_p^\natural(E/\Q)$
\end{enumerate}\end{conjecture}

\remark The term $\Reg_p^\natural(E/\Q)$ is independent from the choice of Weierstra\ss{ }equation. This follows from the proof of part 2 of Theorem \ref{belowthisone}, which only compares \textit{coordinates} with respect to the basis $\nu_\alpha,\nu_\beta$.

\begin{theorem}\label{belowthisone}This conjecture is equivalent to that of Bernardi and Perrin-Riou (i.e. Conjecture \ref{bernardiandperrinriou}).
\end{theorem}

\begin{definition}Let $r=r(E)>0$. Given a vector $\nu\in D_p(E)$ (or in $D_p(E)(\alpha))$ that is not a linear multiple of $\omega$, we put \begin{linenomath}$$\widetilde{\Reg}_{\nu}:=\frac{\Reg_\nu}{[\omega,\nu]^{r-1}}.$$\end{linenomath}
\end{definition}
\remark We know that $\widetilde{\Reg}_{\nu}$ is linear in $\nu$. See e.g. \cite[proof of Lemma 4.2]{steinwuthrich}.

\begin{proof}[Proof of equivalence for part $1$] This follows from Lemma \ref{ordersofvanishingagree} and the product rule.
\end{proof}

\begin{proof}[Proof for part $2$]From the equivalence of part $1$ and the product rule, we have \begin{linenomath}$$\overrightarrow{L}_p^*\Log_{\alpha,\beta}(1)=\overrightarrow{L}_p^*Z=(L_{\alpha}^*,{L_{\beta}^*}).$$\end{linenomath} But we also have, for $r>0$, \begin{linenomath}$$(1-\phi)^2(\Reg_{\frac{1}{\beta}},\Reg_{\frac{1}{\alpha}})\hidari\nu_A \\ \nu_B\migi=\left(\frac{\widetilde{\Reg}_{\nu_B}}{[\omega,\nu_B]},\frac{\widetilde{\Reg}_{\nu_A}}{[\omega,\nu_A]}\right)\smat{ (1-\frac{1}{\alpha})^2 & 0 \\ 0 & (1-\frac{1}{\beta})^2}Z^{-1}\hidari \nu_\sharp \\ \nu_\flat \migi.$$\end{linenomath}
Since $[\omega,\nu_B]=[\nu_A,\nu_B]=-[\omega,\nu_A]$ and $\widetilde{\Reg}_{\nu}$ is linear in $\nu$, and by Lemma \ref{nonlinearity}, this is equal to
\begin{linenomath}$$\frac{1}{\det Z} \left(\frac{1}{[\nu_A,\nu_B]}\widetilde{\Reg}_{N_\sharp},
\frac{-1}{[\nu_A,\nu_B]}\widetilde{\Reg}_{N_\flat}\right)\hidari \nu_\sharp \\ \nu_\flat \migi.$$
But $\widetilde{\Reg}_{N_\sharp}=[\omega,N_\sharp]\Reg_\sharp$, so this is equal to 
$$
\left(\frac{[\omega,N_\sharp]}{[\nu_\sharp,\nu_\flat]}\Reg_\sharp, \frac{-[\omega,N_\flat]}{[\nu_\sharp,\nu_\flat]}\Reg_\flat\right)\hidari \nu_\sharp \\ \nu_\flat \migi.$$\end{linenomath} The rest follows from explicit calculation of the factors preceding the regulators.
\end{proof}

What is know so far is the following theorem of Kato:
\theorem[{\cite{kato}, cf. \cite[Theorem 9.4]{kobayashi} when $a_p=0$}]
In Conjectures \ref{tandempadicbsd}, \ref{equivalentformulationofabove}, \ref{bernardiandperrinriou}, and \ref{mazurtateteitelbaum}, the orders of vanishing of the $p$-adic $L$-functions are all $\geq r(E)$.

\subsection{A remark in the ordinary case}\label{ordinarybsd}
When $p$ is ordinary, there is the following conjecture of Mazur, Tate, and Teitelbaum.

\begin{conjecture}[Mazur, Tate, and Teitelbaum]\label{mazurtateteitelbaum} Let $p$ be a good ordinary prime, and denote by $r^{an}$ the order of vanishing of $L_p(E,\alpha,T)$ at $0$, and by $L_p^*(E)$ the leading coefficient of the Taylor expansion at $0$.
\begin{enumerate}
\item We have $r^{an}=r(E)$.
\item $L_p^*(E)=\left(1-\frac{1}{\alpha}\right)^2\frac{\prod_v c_v \cdot \#\Sha(E/\Q)\cdot \Reg_{\frac{1}{\beta}}(E/\Q)}{(\#E(\Q)_{tors})^2}$.
\end{enumerate}
\end{conjecture}

\remark These conjectures are a combination of \cite[\S II.10, Conjecture (BSD($p$))]{mtt}, which asserts that $r^{an}\geq r(E)$, and the remark thereafter, which predicts the non-vanishing of $\Reg_{\frac{1}{\beta}}(E/\Q)$.

\remark We encounter the term $\Reg_{\frac{1}{\beta}}$ (rather than $\Reg_{\frac{1}{\alpha}}$) because of our choice of Frobenius $\phi=\frac{F}{p}$, where $F$ is the Frobenius as chosen in \cite{mazursteintate} or \cite{kedlaya}. The regulator comes from the normalized height associated to the unit-eigenvector $\alpha$ of $F$ on $D_p(E)$, so that the eigenvalue for $\phi$ becomes $\frac{\alpha}{p}=\frac{1}{\beta}$.\footnote{In \cite[Section 4.1]{steinwuthrich}, the regulator was accidentally constructed from the height coming from the normalized eigenvector of $\phi$ with eigenvalue $\frac{1}{\alpha}$. Everything works in that section if one replaces $\alpha$ by $\beta$.}

In the ordinary case, recall that $L_\sharp$ and $L_\flat$ are not well-defined, but their values at $T=0$ are. In particular this means that when $a_p=2$ and $p$ is odd, we may have $L_\flat(E,0)=0$ while $L(E,1)\neq0$, reminiscent of an extra zero phenomenon. This leads us to ask:

\begin{question} \textit{Where does this extra-zero phenomenon come from?}
\end{question}


\rm

\section{The greatest common divisor}
We now generalize and give some evidence for the following conjecture found in \cite[Problem 3.2]{kuriharapollack}.

Recall that $d_n$ denoted the normalized jump in the ranks at the $n$th level of the cyclotomic tower:
$$ d_n=\frac{1}{p^n-p^{n-1}}(\rank E(\Q_n)- \rank E(\Q_{n-1}))$$

\begin{conjecture}[The problem of Kurihara and Pollack]\label{kuripola}

Let $E/\Q$ be an elliptic curve so $p$ is an odd prime of good supersingular reduction and $a_p=0$. Then
$$\gcd(L_\sharp(E,T),L_\flat(E,T))=\left(T^{r}\prod_{ d_n \geq1 \text{ and } n\geq 1}\Phi_{p^n}^{d_n-1}(1+T)\right).$$
\end{conjecture}

\rm Note that this is an equality of {\it ideals}, since the greatest common divisor of two functions of ${\Q\otimes \Z_p[[T]]}$ is only well-defined as a $\Z_p[[T]]$-ideal. We can give the following proposition:
\begin{proposition}\label{kuriharapollack}Let $E/\Q$ be an elliptic curve and $p$ a prime of good supersingular reduction. For some polynomial $P_\Sha(E,T)$ with $P_\Sha(E,\zeta_{p^n}-1)\neq0$ for $n\geq0,$
$$\gcd\left(L_\sharp(E,T),L_\flat(E,T)\right)=\left(P_\Sha(E,T)\cdot T^{r^{an}}\prod_{ d_n \geq1 \text{ and } n\geq 1}\Phi_{p^n}^{\epsilon_n^{an}-1}(1+T)\right),$$
where $\epsilon_n^{an}-1=d_n^{an}-1$ or $\epsilon_n^{an}-1=d_n^{an}.$
\end{proposition}

\convention Given a vector $(f(T),g(T))$ of $p$-adic analytic functions, we define its order of vanishing at $s$ by $\ord_{T=s}(f(T),g(T)):=\min(\ord_{T=s}f(T),\ord_{T=s}g(T))$.

\begin{lemma}\label{lemma1}
Denote by $\iota$ (any) complex conjugation. Let $f(T), g_1(T)$, $g_2(T)$, and the entries of a $2\times2$ matrix $M(T)$ be $p$-adic analytic functions on the open unit disc and $e=\ord_{T=s}f(T)$ so that 
$$\left(f(T),\iota\left(f(T)\right)\right)=(g_1(T),g_2(T))M(T)$$
and $\det M(s) \neq 0$. Then we have $\ord_{T=s}(g_1(T),g_2(T))=e$.
\end{lemma}
\begin{proof}By calculus, $\left(f^{(m)}(s),\iota\left({f^{(m)}}(s)\right)\right)=(0,0)$ if and only if $(g_1^{(m)}(s),g_2^{(m)}(s))=(0,0)$ for ${m \geq 0}$.
\end{proof}
\rm
\begin{corollary}\label{lemma3}
The exact power of $T$ dividing $\gcd\left(L_\sharp(E,T),L_\flat(E,T)\right)$ is $T^{r^{an}}$.
\end{corollary}

\begin{lemma}\label{lemma2}
Let $\overrightarrow{f}(T)=(f_1(T),f_2(T)) \text{ and }\overrightarrow{g}(T)=(g_1(T),g_2(T))$ be vectors of analytic functions on the open unit disc satisfying 
$$\overrightarrow{f}(T)=\overrightarrow{g}(T)\\\CCC_n,\text{  where } \CCC_n= \links a_p & p \\ -\Phi_{p^n}(1+T) & 0 \rechts.$$
Let $s=\zeta_{p^n}-1$. Then $\ord_{T=s}\overrightarrow{g}(T)=\ord_{T=s}\overrightarrow{f}(T)$ or $\ord_{T=s}\overrightarrow{g}(T)=\ord_{T=s}\overrightarrow{f}(T)-1$.
\end{lemma}
\begin{proof}Since $\ord_{T=s} g_1(T)=\ord_{T=s} f_2(T)$ and {$a_pg_1(T)-f_1(T)=-\Phi_{p^n}(1+T)g_2(T)$},
$$\begin{cases}\ord_{T=s} f_1(T)< \ord_{T=s} f_2(T)& \text{ implies } \ord_{T=s} g_2(T)=\ord_{T=s} f_1(T)-1,\\
\ord_{T=s} f_1(T)= \ord_{T=s} f_2(T)\text{ and $a_p\neq0$} &\text{ implies } \ord_{T=s} g_2(T)\geq\ord_{T=s} f_1(T)-1,\\
 \ord_{T=s} f_1(T)> \ord_{T=s} f_2(T)\text{ and $a_p\neq0$} &\text{ implies }\ord_{T=s} g_2(T)=\ord_{T=s} f_2(T)-1,\\
  \ord_{T=s} f_1(T)\geq \ord_{T=s} f_2(T)\text{ and $a_p=0$} &\text{ implies } \ord_{T=s}g_2(T)=\ord_{T=s} f_1(T)-1.\end{cases}$$
\end{proof}

\begin{proof}[Proof of Proposition \ref{kuriharapollack}] We use Corollary $\ref{lemma3}$ and the following argument: Let $\mathcal{M}=I$ when $n=1$ and $\mathcal{M}=\CCC_1\cdots\CCC_{n-1}$ when $n>1$. Recall that $\overrightarrow{L}_p=\left(L_\sharp(E,T),L_\flat(E,T)\right)$, so that
$$\left(L_p(E,\alpha,T),L_p(E,\beta,T)\right)=\overrightarrow{L}_p\mathcal{M}\CCC_n\Xi_n$$
for some $2\times2$ matrix $\Xi_n$ so that $\det \Xi_n(\zeta_{p^n}-1)\neq0$. From Lemma \ref{lemma1}, ${\ord_{T={\zeta_{p^n}-1}}\left(\overrightarrow{L}_p\mathcal{M}\CCC_n\right)=d_n^{an}}$. Lemma \ref{lemma2} then implies $${\ord_{T={\zeta_{p^n}-1}}\left(\overrightarrow{L}_p\mathcal{M}\right)=d_n^{an}-1}\text{ or }{\ord_{T={\zeta_{p^n}-1}}\left(\overrightarrow{L}_p\mathcal{M}\right)=d_n^{an}}.$$ From $\det \mathcal{M} (\zeta_{p^n}-1)\neq0$ and Lemma \ref{lemma1} again, ${\ord_{T={\zeta_{p^n}-1}}\overrightarrow{L}_p=d_n^{an}-1}$ or ${\ord_{T={\zeta_{p^n}-1}}\overrightarrow{L}_p=d_n^{an}}$.
\end{proof}
In view of the problem of Kurihara and Pollack (Conjecture \ref{kuripola}), we make the following conjecture:
\begin{conjecture}Let $E/\Q$ be an elliptic curve, and $p$ a good supersingular prime. Then
$$\gcd(L_\sharp(E,T),L_\flat(E,T))=\left(T^{r}\prod_{ d_n \geq1 \text{ and } n\geq 1}\Phi_{p^n}^{d_n-1}(1+T)\right).$$
\end{conjecture}

\small Florian Sprung, Princeton University and the Institute for Advanced Study, Princeton, NJ. \newline
email: fsprung@math.princeton.edu


\begin{thebibliography}{PTW02}
\bibitem[BPR93]{bpr}D. Bernardi, B. Perrin-Riou: \textit{Variante $p$-adique de la conjecture de Birch et Swinnerton-Dyer (le cas supersingulier)}, Comptes Rendus de l'Acad\'{e}mie des Sciences. Paris S\'{e}rie I Math\'{e}matique \textbf{317} (1993), no. 3, 227-232.
\bibitem[KP07]{kuriharapollack}M. Kurihara, R. Pollack: \textit{Two $p$-adic $L$-functions and rational points on elliptic curves with supersingular reduction}, $L$-functions and Galois representations, 300-332.
\bibitem[Co04]{colmez} Colmez, P.: \textit{La conjecture de Birch et Swinnerton-Dyer $p$-adique}, Ast\'{e}risque \textbf{294, ix} (2004), 251-319.
\bibitem[Ka04]{kato}K. Kato: \textit{$p$-adic Hodge theory and values of zeta functions of modular forms}, Ast\'{e}risque \textbf{295} (2004), 117-290.
\bibitem[Ke01]{kedlaya}K. Kedlaya: \textit{Counting points on hyperelliptic curves using Monsky-Washnitzer cohomology}, Journal of the Ramanujan Mathematical Society, \textbf{16} (2001), no. 4, 323-338.
\bibitem[Ko03]{kobayashi}
S. Kobayashi: \textit{Iwasawa theory for elliptic curves at supersingular primes}, Invent. Math. \textbf{152}  (2003), no.1, 1-36.
\bibitem[LZ13]{loefflerzerbes} D. Loeffler, S. Zerbes, J. Reine Angew. Math., 679 (2013), p. 181 - 206.
\bibitem[MST06]{mazursteintate} B. Mazur, W. Stein, J. Tate: \textit{Computation of $p$-adic heights and log convergence}, Doc. Math. 2006, Extra Vol., 577-614.
\bibitem[MTT86]{mtt}B. Mazur, J. Tate, and J. Teitelbaum: \textit{On p-adic analogues of the conjectures of Birch and Swinnerton-Dyer}, Invent. Math. \textbf{84} (1986), 1-48.
\bibitem[PR03]{perrinriou} B. Perrin-Riou: \textit{Arithm\'{e}tique des courbes elliptiques \`{a} r\'{e}duction supersinguli\`{e}re.} Experiment. Math. \textbf{12} (2003), 155-186.
\bibitem[Po03]{pollack} R. Pollack: \textit{The $p$-adic $L$-function of a modular form at a supersingular prime}, Duke Math. J. \textbf{118} (2003), no. 1, 1-36.
\bibitem[Sil]{silverman} J. Silverman: \textit{The arithmetic of elliptic curves}, Second Edition, Graduate Texts in Mathematics \textbf{106} (2009), Springer, New York. 
\bibitem[Sp12]{shuron}F. Sprung: \textit{Iwasawa Theory for elliptic curves at supersingular primes: A pair of Main Conjectures}, Journal Of Number Theory \textbf{132} (2012), no. 7. 
\bibitem[Sp13]{sha3}F. Sprung: \textit{The \v{S}afarevi\v{c}-Tate group of an elliptic curve in cyclotomic $\Z_p$-textensions at supersingular primes} J. Reine Angew. Math., \textbf{681} (2013), August 2013.
\bibitem[Sp15]{surprisesha}F. Sprung: \textit{On pairs of $p$-adic $L$-functions for weight two modular forms}, Submitted.
\bibitem[Sp16]{rank}F. Sprung: \textit{The rank of an elliptic curve in cyclotomic $\Z_p$-extensions at supersingular primes}, in preparation.
\bibitem[SW13]{steinwuthrich}W. Stein, C. Wuthrich: \textit{Algorithms for the Arithmetic of Elliptic Curves using Iwasawa Theory}, Mathematics of Computation, \textbf{82} (2013), 1757-1792.

\end{thebibliography}
\end{document}